\documentclass[11pt, amsfonts]{amsart}

\usepackage{amsmath,amssymb,amsthm,mathtools}
\usepackage[all]{xy}
\usepackage[usenames]{color}

\definecolor{darkgreen}{RGB}{65,165,118}
\definecolor{darkred}{RGB}{180,0,0}

\newcommand{\hadgesh}[1]{\emph{\color{darkred}#1}}

\usepackage[colorlinks=true, linkcolor=blue, citecolor=magenta]{hyperref}
\usepackage{tikz}

\numberwithin{equation}{section}

\SelectTips{eu}{12}

\newtheorem{theorem}{Theorem}[section]
\newtheorem{proposition}[theorem]{Proposition}
\newtheorem{lemma}[theorem]{Lemma}
\newtheorem{corollary}[theorem]{Corollary}

\theoremstyle{definition}
\newtheorem{definition}[theorem]{Definition}
\newtheorem{example}[theorem]{Example}

\newtheorem{Th}{Theorem}

\theoremstyle{remark}
\newtheorem{remark}[theorem]{Remark}

\newcommand{\Z}{\mathbb{Z}}

\newcommand{\lk}{\mathrm{lk}}
\newcommand{\dl}{\mathrm{dl}}
\newcommand{\xto}[1]{\xrightarrow{#1}}
\newcommand{\proj}{\mathrm{proj}}
\newcommand{\hofib}{\mathrm{hofib}}
\newcommand{\Ker}{\mathrm{Ker}}
\newcommand{\coKer}{\mathrm{Coker}}

\newcommand{\Ima}{\mathrm{Im}}
\newcommand{\PP}{\mathbf{P}}
\newcommand{\F}{\mathbf{F}}
\newcommand{\K}{\mathbf{K}}
\newcommand{\E}{\mathbf{E}}
\newcommand{\C}{\mathbf{C}}

\newcommand{\Fl}{\mathrm{Fl}}

\newcommand{\QQ}{\mathbb{Q}}

\title[Polyhedral products with finite generalised Postnikov decomposition]{Characterisation of polyhedral products with finite generalised Postnikov decomposition}

\author{Kouyemon Iriye}
\address{Department of Mathematical Sciences, Osaka Prefecture University, Sakai, 599-8531, Japan}
\email{kiriye@mi.s.osakafu-u.ac.jp}

\author{Daisuke Kishimoto}
\address{Department of Mathematics, Kyoto University, Kyoto, 606-8502, Japan}
\email{kishi@math.kyoto-u.ac.jp}

\author{Ran Levi}
\address{Institute of Mathematics, University of Aberdeen, Aberdeen, AB24 3UE, UK}
\email{r.levi@abdn.ac.uk}


\begin{document}

\baselineskip.525cm

\maketitle

\begin{abstract}
  A generalised Postnikov tower for a space $X$ is a tower of principal fibrations with fibres generalised Eilenberg-MacLane spaces, whose inverse limit is weakly homotopy equivalent to $X$. In this paper we give a characterisation of a polyhedral product $Z_K(X,A)$ whose universal cover either admits a generalised Postnikov tower of finite length, or is a homotopy retract of a space admitting such a tower. We also include $p$-local and rational versions of the theorem. We end with a group theoretic application.
\end{abstract}



\section{Introduction}

Let $(X,A)$ be a pair of topological spaces and let $K$ be a simplicial complex. The polyhedral product $Z_K(X,A)$ is a space obtained by forming Cartesian products of the spaces $X$ and $A$, one for each simplex in $K$, and gluing them together according to the face relations among the simplices in $K$.  Polyhedral products were introduced by Bahri, Bendersky, Cohen, and Gitler \cite{BBCG} as a generalization of two fundamental constructions in toric topology: the moment-angle complex and the Davis-Januszkiewicz space. Similar constructions can be found in earlier works in algebraic topology \cite{A,P,KT}. Polyhedral products are  combinatorially defined objects and recent studies show that the interaction between topology and combinatorics through polyhedral products provides a rich source of ideas and applications \cite{BBCG,GT2,GW,IK1,IK2,KL}. A comprehensive survey can be found in \cite{BBC}.

It is particularly interesting to understand how a given homotopical property of $Z_K(X,A)$ is related to combinatorial and homotopical properties of  $K$ and $(X,A)$. For example, combinatorial conditions on $K$, which guarantee that the polyhedral product $Z_K(CX,X)$ (here $CX$ means the cone on $X$) is a suspension, are studied in \cite{GT2,GW,IK1,IK2}.

In this paper, we study polyhedral products $Z_K(X,A)$ that admit a Postnikov type decomposition satisfying certain finiteness properties, and characterise such spaces in terms of properties of the simplicial complex $K$ and the pair $(X,A)$.  We start by introducing the largest class of spaces to which our results apply. Throughout this paper all spaces are assumed to be of the homotopy types of a CW-complex.

\begin{definition}\label{Def-PP}
  Let $\PP$ denote the minimal class that contains all connected spaces, whose universal covers admit finite type generalised Postnikov towers of finite length, and that is closed under homotopy retracts.
\end{definition}
See Section \ref{Postnikov} for the definition of generalised Postnikov towers.  (See also \cite{IK}.) Spaces in $\PP$ are sometimes referred to in the literature as polyGEMs \cite{CFFS}.

We say that a connected space $X$ is \hadgesh{of finite type} if $\widetilde{H}_i(X;\Z)$ is a finitely generated abelian group for each $i\ge 1$. Thus, a simply-connected space $X$ is of finite type if and only if $\pi_i(X)$ is a finitely generated abelian group for each $i>1$. Note that the universal cover of each space in $\PP$ is of finite type.

Our results become a bit more specific for certain subclasses of $\PP$.

\begin{definition}
	\begin{enumerate}
		\item Let $\F$ be the class of all connected spaces whose universal covers are of finite type and have non-trivial homotopy groups in finitely many dimensions.
		\item Let $\K$ be the class of all connected spaces whose universal covers are of finite type and have finitely many non-trivial $k$-invariants.
	\end{enumerate}

\end{definition}

We consider the question what are necessary and sufficient conditions for a polyhedral product $Z_K(X,A)$ to belong to each one of these classes, depending on the pair $(X,A)$ and the simplicial complex $K$. We are now ready to state our main theorem.

\begin{Th}
  \label{main}
  Let $K$ be a simplicial complex, and let $(X,A)$ be an NDR-pair. Let $F$ denote the homotopy fibre of the inclusion $A\to X$, and assume that each connected component of $F$ is of finite type. Then $Z_K(X,A)$ belongs to $\PP$ if and only if one of the following conditions holds:
  \begin{enumerate}
    \item $X\in\PP$ and $K$ is a simplex;
    \item $X\in\PP$, $F$ is an acyclic space, and $K$ is not a simplex;
    \item $X\in\PP$, $F$ is a disjoint union of acyclic spaces, and $K$ is a flag complex that is not a simplex.
  \end{enumerate}
Furthermore, the statement above holds for $\PP$ replaced throughout by $\F$ or by $\K$.
\end{Th}

The question of when a polyhedral product of the form $Z_K(X,*)$ is a $K(\pi,1)$ was treated in \cite{K,BBC}  (see Lemma \ref{K(pi,1)} below). The proof there does not generalise to our context. However the result for this special case is used in the proof of the theorem (see the proof of Theorem \ref{Z(CX,X) p-local}).

Let $p$ be a prime or zero. By analogy to the integral case, we say that a connected space $X$ is of finite $p$-local type if $\widetilde{H}_n(X;\Z)$ is a finitely generated $\Z_{(p)}$-module for each $n>0$. Define a $p$-local analog of the classes $\PP,\F$ and $\K$ by requiring that all spaces are $p$-local spaces  of finite $p$-local type. Let $\PP_p,\F_p$ and $\K_p$ denote the resulting classes respectively. We will use the term "rational" instead of ``$0$-local'' as usual.

\begin{Th}
  \label{main p-local}
  Let $p$ be a prime and $\C_p$ denote one of the classes $\PP_p, \F_p$ or $\K_p$. Let $K$ be a simplicial complex, and let $(X,A)$ be an NDR-pair. Suppose that each connected component of the homotopy fibre $F$ of the inclusion $A\to X$ is of finite $p$-local type. Then $Z_K(X,A)$ belongs to $\C_p$ if and only if one of the following conditions holds:
  \begin{enumerate}
    \item $X\in\C_p$ and $K$ is a simplex;
    \item $X\in\C_p$, $F$ is $p$-locally acyclic, and $K$ is not a simplex;
    \item $X\in\C_p$, $F$ is a disjoint union of $p$-locally acyclic spaces, and $K$ is a flag complex that is not a simplex.
  \end{enumerate}
\end{Th}

For a simply-connected space $X$ whose rational homology is finite dimensional, the well known dichotomy theorem of F\'elix, Halperin and Thomas \cite[Theorem 1.1]{FHT1} states that $\pi_*(X)\otimes\QQ$ is either a finite dimensional vector space, in which case $X$ is called \hadgesh{elliptic}, or a graded vector space that grows exponentially with dimension, in which case it is called \hadgesh{hyperbolic}.  In \cite{HST} the authors give necessary and sufficient conditions for $Z_K(X,A)$ to be elliptic. We believe that the statement is incomplete since the case where the homotopy fibre of the inclusion $A\to X$ is disconnected is not considered. This may be because of an implicit assumption that all spaces considered are simply connected. In our discussion below no simple connectivity assumption is made.

\begin{definition}
Let $\E_0$ denote the class of connected spaces $X$ such that the universal cover of $X$ is a rational elliptic space, i.e. such that $\widetilde{X}$ is rational and has finite dimensional rational homotopy  and rational homology.
\end{definition}

Now we state the rational analog of Theorem \ref{main}, including amendment of \cite[Theorem 1.2]{HST} for elliptic spaces. We say that a space $X$ is a homology rational sphere if $X$ is connected and $\widetilde{H}_*(X;\Z)\cong\mathbb{Q}$.

\begin{Th}
  \label{main rational}
 Let $\C_0$ denote one of the classes $\PP_0, \K_0, \F_0$ or $\E_0$. Let $K$ be a simplicial complex, $(X,A)$ be an NDR-pair, and let $F$ denote the homotopy fibre of the inclusion $A\to X$. Suppose that each connected component of $F$ is rational. Then $Z_K(X,A)$ belongs to $\C_0$ if and only if one of the following conditions holds:
  \begin{enumerate}
    \item $X\in\C_0$ and $K$ is a simplex;
    \item $X\in\C_0$, $F$ is a rationally acyclic space, and $K$ is not a simplex;
    \item $X\in\C_0$, $F$ is a disjoint union of rationally acyclic spaces, and $K$ is a flag complex that is not a simplex;
    \item $X\in\C_0$, $F$ is a homology rational sphere, and minimal non-faces of $K$ are mutually disjoint.
  \end{enumerate}
\end{Th}

Let $\Gamma$ be a simple graph on $m$ vertices with the edge set $E$, and let $\mathbf{G} = \{G_i\}_{i=1}^m$ be a collection of discrete groups. The graph product of $\mathbf{G}$ over $\Gamma$ is defined by
\[\mathbf{G}^\Gamma= \langle G_1, \ldots ,G_m]\;|\; [G_i,G_j]=1\;\forall\{i,j\}\in E\rangle.\]
Let $K_\Gamma(\mathbf{G})$ denote the \hadgesh{graph product kernel of $\mathbf{G}$ over $\Gamma$}, i.e., the kernel of the canonical map $\mathbf{G}^\Gamma\to \prod_{i=1}^m G_i$.

Part of the argument used in the proof of  Theorem \ref{main} is used  to obtain a short, purely homotopy theoretic proof of \cite[Theorem 12 (1)]{Kim} (See Remark \ref{Rem-discrete-group-trick}). In our version the assumption that the groups involved are countable is dropped.

\begin{Th}
	\label{application}
	Let $\mathbf{G}$ and $\mathbf{H}$ be two collections of $m$ discrete groups, and let $\Gamma$ be a simple graph on $m$ vertices. Assume that for each $1\le i\le m$, $|G_i|$ and $|H_i|$ have the same cardinality. Then
	\[K_\Gamma(\mathbf{G}) \cong K_\Gamma(\mathbf{H}).\]
\end{Th}

The paper is organised as follows. In Section \ref{Postnikov} we discuss some  properties of generalised Postnikov towers. In Section \ref{Polyhedral} we recall some basic properties of polyhedral products, and in Section \ref{Polyhedral-CX-X} we specialise to $Z_K(CX, X)$ and prove the main technical results of the paper, Theorems \ref{Z(CX,X) p-local}, \ref{Z(CX,X)} and \ref{Z(CX,X) rational},  that are used in Section \ref{Main} to prove Theorems \ref{main} and \ref{main p-local}. Finally in Section \ref{Groups} we prove Theorem \ref{application}.

\textit{Acknowledgement:} The authors would like to thank Ashot Minasyan for informing them of the paper \cite{Kim} on graph products of groups. Iriye and Kishimoto are supported by JSPS KAKENHI No. 19K03473 and No. 17K05248. Levi is supported in part by an EPSRC grant  EP/P025072/1.



\section{The class $\PP$}
\label{Postnikov}

In this section, we study some properties of the class $\PP$ relevant for our purpose. We first recall from \cite{IK} the definition of generalised Postnikov tower. A generalised Eilenberg-MacLane space is a (possibly infinite) product of ordinary Eilenberg-MacLane spaces. Generalised Eilenberg-MacLane spaces will be referred to  as a GEMs, for short.
A \hadgesh{generalised Postnikov tower} for a space $X$ is a sequence of spaces and maps:
$$\cdots\to X_n\to X_{n-1}\to\cdots\to X_0$$
satisfying the following conditions:
\begin{enumerate}
  \item $X_0$ is weakly contractible;
  \item for each $n$, there is a homotopy fibration $X_n\to X_{n-1}\xto{k_n} K_n$, where $K_n$ is a simply-connected GEM;
  \item there is a weak homotopy equivalence $X\to\mathrm{holim}\,X_n$.
\end{enumerate}

The classical Postnikov tower for a simply-connected space $X$ is obtained where $K_n = K(\pi_{n}(X), n+1)$. In that case the maps $k_n\colon X_{n-1}\to K_n$ are referred to as the \hadgesh{$k$-invariants} for $X$. Postnikov towers can also be defined more generally for nilpotent spaces, but we will not require this level of generality in this paper.

The following lemma is well known and we include a proof for the convenience of the reader.

\begin{lemma}\label{lem-P0P1P2}
	There are inclusions
  $$\F\subset\K\subset\PP,\quad\F_p\subset\K_p\subset\PP_p,\quad\E_0\subset\F_0\subset\K_0\subset\PP_0.$$
\end{lemma}

\begin{proof}
	We only prove the first inclusions because other cases are similar. Clearly $\F\subseteq \K$. Suppose $X\in \K$ and let
  $$\cdots\to X_n\to X_{n-1}\to\cdots\to X_0$$
  be the classical Postnikov tower for its universal cover $\widetilde{X}$. If the $k$-invariant $k_n\colon X_{n-1}\to K(\pi_n(X),n+1)$ is trivial, then $X_n\simeq X_{n-1}\times K(\pi_n(X),n)$.  Thus there is an integer $N$ such that the $k$-invariant $k_n$ is trivial for each $n>N$, implying there is a weak homotopy equivalence $\widetilde{X}\to X_N\times\prod_{n>N}K(\pi_n(X),n)$. Thus a tower
  $$X_N\times\prod_{n>N}K(\pi_n(X),n)\to X_N\to X_{N-1}\to\cdots\to X_0$$
  is a finite type generalised Postnikov tower of finite length for $\widetilde{X}$. Therefore $X\in\PP$, completing the proof.
\end{proof}

\begin{lemma}
	The class $\PP$ is closed under finite products.
	\label{lem-products}
\end{lemma}
\begin{proof}
	Assume $X, Y\in\PP$, and let $\widetilde{X}$ and $\widetilde{Y}$ denote their respective universal covers. Then there are spaces $X_0$ and $Y_0$ admit finite type generalised Postnikov towers of finite length, and such that $\widetilde{X}$ and $\widetilde{Y}$ are homotopy retracts of $X_0$ and $Y_0$, respectively. Thus $\widetilde{X}\times\widetilde{Y}$ that is the universal cover of $X\times Y$ is a homotopy retract of $X_0\times Y_0$, and the latter clearly admits a finite type generalised Postnikov tower of finite length. Thus $X\times Y\in \PP$.
\end{proof}

The following properties of the classes $\F,\K$ are obvious, and so we omit the proof.

\begin{lemma}
  \label{F-K}
  Let $\C$ be either of $\F,\K$. Then the following hold:
  \begin{enumerate}
    \item $\C$ is closed under finite products;
    \item $\C$ is closed under homotopy retracts;
    \item If there is a homotopy fibration $F\to E\to B$ such that $B\in\C$ and $F\in\F$, then $E\in\C$.
  \end{enumerate}
\end{lemma}

Next, we need to show that if $F\to E\to B$ is a homotopy fibration, and $B, E\in \PP$, then $F$ also belongs to $\PP$. The following three lemmas will be needed to do so. Let $\hofib(f)$ denote the homotopy fibre of a map $f\colon X\to Y$.

\begin{lemma}
	\label{connectivity}
	Let   $$F\to E\xto{p} B,$$
	be a homotopy fibration of $(r-1)$-connected spaces, $r\geq 2$, and assunme that $B$ is a GEM.
	Let $f\colon F\to K(G,r)$ be any map that induces a surjection on $\pi_r$. Then there are a map $g\colon E\to K(H,r)$ and a homotopy fibration
	$$\mathrm{hofib}(f)\to\mathrm{hofib}(g)\to W$$
  such that $g$ induces a surjection on $\pi_r$ and $W$ is a GEM. Moreover, if $B$ is of finite type then so is $W$.
\end{lemma}

\begin{proof}
	Write $B=\prod_{n\ge r}K(A_n,n)$, and let $h$ denote the composite
	\[E\xto{p} B = \prod_{n\ge r}K(A_n,n)\xto{\proj} K(A_r,r).\]
	Then $h$ is surjective in $\pi_r$, since $F$ is $(r-1)$-connected, and one easily verifies that the  sequence
	\[F\xto{e} \hofib(h)\xto{} \prod_{n\ge r+1}K(A_n,n)\]
	is a homotopy fibration.
	In particular, $\mathrm{hofib}(h)$ is $(r-1)$-connected and there is an exact sequence of abelian groups:
	\begin{equation}
	\label{exact}
	A_{r+1}\to\pi_r(F)\xrightarrow{e_*}\pi_r(\mathrm{hofib}(h))\to 0.
	\end{equation}

	Let $I$ and $J$ denote $\Ker(e_*)$ and $\Ker(f_*)$ respectively. Expand the diagram
	\[G\xleftarrow{f_*} \pi_r(F)\xto{e_*}\pi_r(\hofib(h))\]
	into a commutative diagram of  abelian groups with exact rows and columns:
	$$\xymatrix{&0\ar[d]&0\ar[d]&0\ar[d]\\
		0\ar[r]&I\cap J\ar[r]\ar[d]&J\ar[r]\ar[d]&e_\ast(J)\ar[r]\ar[d]&0\\
		0\ar[r]&I\ar[r]\ar[d]&\pi_r(F)\ar[r]^(.4){e_\ast}\ar[d]^{f_\ast}&\pi_r(\mathrm{hofib}(h))\ar[r]\ar[d]&0\\
		0\ar[r]&I/I\cap J\ar[r]\ar[d]&G\ar[r]\ar[d]&K\ar[r]\ar[d]&0\\
		&0&0&0}$$
	In particular,
	by \eqref{exact}, $I$ is a quotient of $A_{r+1}$. Then since both $F$ and $\mathrm{hofib}(h)$ are $(r-1)$-connected, one gets a homotopy commutative diagram whose rows are homotopy fibrations
	$$\xymatrix{
		F\ar[r]^(.45)e\ar[d]\ar@/_4.0pc/[dd]_f&\mathrm{hofib}(h)\ar[d]^\alpha\ar[r]&\prod_{n\ge r+1}K(A_n,n)\ar[d]^q\\
		K(\pi_r(F),r)\ar[r]\ar[d]&K(\pi_r(\mathrm{hofib}(h)),r)\ar[r]\ar[d]^\beta&K(I,r+1)\ar[d]\\
		K(G,r)\ar[r]&K(K,r)\ar[r]&K(I/I\cap J,r+1)
	}$$
	where $q$ is the composite
	\[\prod_{n\ge r+1}K(A_n,n)\xto{\proj} K(A_{r+1},r+1)\to K(I,r+1).\]

	Now, set $h' = \beta\circ\alpha$. Then there is a homotopy fibration
	$$\mathrm{hofib}(f)\to\mathrm{hofib}(h')\to\prod_{n\ge r+1}K(A_n',n)\eqqcolon W,$$
	where $A_{r+1}'=\Ker(A_{r+1}\to I\to I/I\cap J)$ and $A_n'= A_n$ for $n\ge r+2$. Let $\gamma$ denote the composite  $\mathrm{hofib}(h')\to\mathrm{hofib}(h)\to E$. Then, by a simple diagram chase,  $\hofib(\gamma) \simeq K(H,r-1)$, where $H$ is an extension of $A_r$ by $K$.  Since $E$ is $(r-1)$-connected, the homotopy fibration $K(H,r-1)\to\mathrm{hofib}(h')\xto{\gamma} E$ is principal. Thus, by delooping that fibration, one obtains a map $g\colon E\to K(H,r)$. A diagram chase  now shows that $g$  induces a surjection on $\pi_r$,  and that $\mathrm{hofib}(g)\simeq\mathrm{hofib}(h')$.
	Thus we obtain a homotopy fibration
	\[\hofib(f)\to \hofib(g)\to W\]
	as claimed. By construction, $W$ is clearly of finite type whenever so is $B$.
\end{proof}

\begin{lemma}
	\label{GPT}
	If $X$ is $(r-1)$-connected for $r\ge 2$ that admits a  generalised Postnikov tower of length $l$ (and finite type), then $X$ has a generalised Postnikov tower
	$$X\simeq X_l\to\cdots\to X_n\to X_{n-1}\to\cdots\to X_0$$
	of length $l$ (and finite type), such that for each $1\le n\le l$ there is a homotopy fibration $X_n\to X_{n-1}\to K_n$, where $K_n$ is an $r$-connected GEM.
\end{lemma}

\begin{proof}
	Suppose $X$ has a generalised Postnikov tower:
	$$X\simeq Y_l\to\cdots\to Y_n\to Y_{n-1}\to\cdots\to Y_0$$
	of length $l$. By \cite[Lemma 4.2]{IK}, we may assume that for each $1\le n\le l$, $Y_n$ is $(r-1)$-connected and there is a homotopy fibration $Y_n\to Y_{n-1}\to L_n$, where $L_n$ is an $(r-1)$-connected GEM. Let $f_l\colon Y_l\to*$ be the constant map. Lemma \ref{connectivity} applied to $f_l$ and the homotopy fibration $Y_l\to Y_{l-1}\to L_l$ yields a map $f_{l-1}\colon Y_{l-1}\to K(G_{l-1},r)$ which is surjective in $\pi_r$ and a homotopy fibration
	\[X\simeq Y_l=\mathrm{hofib}(f_l)\to\mathrm{hofib}(f_{l-1})\to K_l,\]
	where $K_l$ is an $r$-connected GEM. Set $X_l = X$ and $X_{l-1} = \hofib(f_{l-1})$. Next consider the homotopy fibration $Y_{l-1}\to Y_{l-2}\to L_{l-1}$ and the map $f_{l-1}$. By Lemma  \ref{connectivity} again, one obtains a map $f_{l-2}\colon Y_{l-2}\to K(G_{l-2},r)$ and a homotopy fibration
	\[X_{l-1} \to\hofib(f_{l-2}) \to K_{l-1},\]
	where $K_{l-1}$is an $r$-connected GEM. Set $X_{l-2} = \hofib(f_{l-2})$.

	By  induction, for $n\le l$ one has  a map $f_n\colon Y_n\to K(G_n,r)$ that is surjective in $\pi_r$, with $X_n= \hofib(f_n)$. Applying Lemma \ref{connectivity} to $f_n$ and the fibration $Y_n\to Y_{n-1}\to L_{n-1}$, we obtain a map $f_{n-1}\colon Y_{n-1}\to K(G_{n-1}, r)$, and a homotopy fibration
	\[X_n\to X_{n-1}\to K_n,\] where $X_{n-1} = \hofib(f_{n-1})$ and $K_n$ is an $r$-connected GEM. Since $Y_0$ is weakly contractible and $X$ is $(r-1)$-connected, $Y_1$ is an $(r-1)$-connected GEM, and it follows that $X_1$ is also an $(r-1)$-connected GEM. Hence taking $X_0 =*$ yields the desired Postnikov tower for $X$. Finite type is clearly preserved throughout this construction.
\end{proof}

A sequence $X_1\to X_2\to X_3$ is said to be a homotopy retract of a sequence $Y_1\to Y_2\to Y_3$ if there is a homotopy commutative diagram
$$\xymatrix{X_1\ar[r]^{i_1}\ar[d]&Y_1\ar[r]^{r_1}\ar[d]&X_1\ar[d]\\
X_2\ar[r]^{i_2}\ar[d]&Y_2\ar[r]^{r_2}\ar[d]&X_2\ar[d]\\
X_1\ar[r]^{i_3}&Y_3\ar[r]^{r_3}&X_3}$$
such that $r_k\circ i_k\simeq 1$ for each $k$.

\begin{lemma}
  \label{retract fibration}
  Let $X$ be a homotopy retract of $X_0$. Then the following statements hold:
  \begin{enumerate}
    \item any homotopy fibration $V\to Y\xrightarrow{p}X$ is a homotopy retract of some homotopy fibration $V_0\to Y\xrightarrow{p_0} X_0$;
    \item any homotopy fibration $W\to X\xrightarrow{q}Z$ is a homotopy retract of some homotopy fibration $W_0\to X_0\xrightarrow{q_0}Z$.
  \end{enumerate}
\end{lemma}

\begin{proof}
  By definition, there are maps $i\colon X\to X_0$ and $r\colon X_0\to X$ such that $r\circ i\simeq 1_X$. Let $p_0=i\circ p\colon Y\to X_0$ and $q_0=q\circ r\colon X_0\to Z$. Then there are homotopy commutative diagrams
  $$\xymatrix{Y\ar@{=}[r]\ar[d]^p&Y\ar@{=}[r]\ar[d]^{\bar{p}}&Y\ar[d]^p\\
  X\ar[r]^i& X_0\ar[r]^r&X}
  \qquad
  \xymatrix{X\ar[r]^i\ar[d]^q& X_0\ar[r]^r\ar[d]^{\bar{q}}&X\ar[d]^q\\
  Z\ar@{=}[r]&Z\ar@{=}[r]&Z}$$
  Thus taking homotopy fibres completes the proof.
\end{proof}

The following elementary proposition is a refinement of \cite[Proposition 4.4]{IK} and is useful for our analysis.

\begin{proposition}
  \label{P}
  Let $F\to E\to B$ be a homotopy fibration of connected spaces, and assume  $E$ and $B$ are in $\PP$. Then so is $F$.
\end{proposition}

\begin{proof}
  For any connected space $X$, let  $\widetilde{X}$ denote its universal cover. By Lemma \ref{retract fibration}, we may assume that $\widetilde{E},\widetilde{B}$ themselves have finite type generalised Postnikov towers of finite length. Consider the homotopy exact sequence:
  $$\cdots\to\pi_{n+1}(B)\xrightarrow{\partial}\pi_n(F)\to\pi_n(E)\to\pi_n(B)\to\cdots$$
  Let $G=\mathrm{Im}\{\partial\colon\pi_2(B)\to\pi_1(F)\}$. Then $G$ is an abelian group and there is a map $\widetilde{B}\to K(G,2)$ which induces $\partial\colon\pi_2(B)\to\pi_1(F)$ on $\pi_2$. Let $\widehat{B}$ be the homotopy fibre of this map. Then $\widehat{B}$ has a generalised Postnikov tower of finite length, and there is a homotopy fibration:
  $$\widetilde{F}\to\widetilde{E}\to\widehat{B}$$

  Let $\widetilde{E}\simeq E_k\to E_{k-1}\to\cdots\to E_0$ and $\widehat{B}\simeq B_l\to B_{l-1}\to\cdots\to B_0$ be finite type generalised Postnikov towers for $\tilde{E}$ and $\widehat{B}$.  By Lemma \ref{GPT}, we may assume that $B_n$ is simply-connected for each $n$ and there is a homotopy fibration $B_n\to B_{n-1}\to K_n$, where $K_n$ is a 2-connected GEM of finite type. For $i\le l$, let $F_i$ be the homotopy fibre of the composite
	$$\widetilde{E}\to\widehat{B}\simeq B_l\to B_{l-1}\to\cdots\to B_i.$$
	Then one gets a sequence of maps
	$$\widetilde{F}\simeq F_l\to F_{l-1}\to\cdots\to F_1\to E_k\to E_{k-1}\to\cdots\to E_0.$$
  We show that this sequence is a finite type generalised Postnikov tower for $\widetilde{F}$. By definition, there is a homotopy fibration $F_1\to E_k\to B_1$ where $B_1$ is a simply-connected GEM. Thus it remains to show that there is a homotopy fibration $F_i\to F_{i-1}\to L_i$ for each $i\le l$ where $L_i$ is a simply-connected GEM of finite type. By definition, there is a homotopy commutative diagram whose rows and columns are homotopy fibrations
  $$\xymatrix{F_i\ar[r]\ar[d]&E_k\ar[r]\ar@{=}[d]&B_i\ar[d]\\
  F_{i-1}\ar[r]\ar[d]&E_k\ar[r]\ar[d]&B_{i-1}\ar[d]\\
  \Omega K_i\ar[r]&\ast\ar[r]&K_i}$$
  for each $i$. Since $\Omega K_i$ is a simply-connected GEM of finite type, the left column of the diagram gives the required fibration, and completes the proof.
\end{proof}

\begin{remark}
  Let $\overline{\PP}$ denote the minimal class that contains all spaces, whose universal covers admit generalised Postnikov towers of finite length (not necessarily of finite type), and that is closed under homotopy retracts. Then $\PP\subseteq\overline{\PP}$, and it is easy to see that Lemma \ref{lem-products} and Proposition \ref{P} hold with $\PP$ replaced by $\overline{\PP}$.
\end{remark}

Let $p$ be a prime number or zero. Recall that if we require $K_n$ to be a simply-connected $p$-local GEM, we obtain the notion of a \hadgesh{$p$-local generalised Postnikov tower}. If each $K_n$ is, of finite $p$-local type, then we say that the $p$-local generalised Postnikov tower is  of finite $p$-local type. Note that Lemma \ref{lem-products} and Proposition \ref{P} hold when $\PP$ is replaced by $\PP_p$.

The following useful lemma is proved in \cite[Theorems 2.5 and 3.1]{IK}, where the assumption that generalised Postnikov towers are of finite type is essential.

\begin{lemma}
  \label{GPT local}
  \begin{enumerate}
    \item A space of finite type has a generalised Postnikov tower of finite type if and only if it is nilpotent.
    \item Let $p$ be a prime number or zero. Suppose that a space $X$ has a generalised Postnikov tower of finite type $\cdots\to X_n\to X_{n-1}\to\cdots\to X_0$ with classifying maps $X_n\to K_n$. Then for any prime $p$, the tower
    $$\cdots\to(X_n)_{(p)}\to(X_{n-1})_{(p)}\to\cdots\to(X_0)_{(p)}$$
    with maps $(X_n)_{(p)}\to(K_n)_{(p)}$ is a $p$-local generalised Postnikov tower  of $p$-local finite type for $X_{(p)}$.
  \end{enumerate}
\end{lemma}

The following is a slight refinement of \cite[Theorem D]{FHL}. Let $\mathrm{cat}\,X$ be the L-S category of a space $X$.

\begin{proposition}
  \label{FHL}
  Let $p$ be a prime and let $X$ be a simply-connected space. If $\mathrm{cat}\,X<\infty$ and $X_{(p)}\in\PP_p$, then $\widetilde{H}_*(X;\Z/p)=0$.
\end{proposition}

\begin{proof}
By definition of  $\PP_p$, there is a simply-connected space $Y$ having a $p$-local generalised Postnikov tower of finite length and $p$-local finite type, such that $X_{(p)}$ is a homotopy retract of $Y$. It follows from \cite[Theorem 6.2]{MS} that $H_*(\Omega Y;\Z/p)$ is a solvable Lie algebra and each $x\in \widetilde{H}_*(\Omega Y;\Z/p)$ satisfies $x^{n}=0$ for some integer $n=n(x)$, possibly depending on $x$. Since $X$ is a homotopy retract of $Y$, $H_*(\Omega X;\Z/p)$ is also solvable as a Lie algebra that is finite dimensional in each degree. The proof now proceeds verbatim as the proof of \cite[Theorem 5.1]{FHL} to show that $\widetilde{H}_*(X;\Z/p)=0$.
\end{proof}

A space $X$ is called $p$-locally acyclic if $\widetilde{H}_*(X;\Z_{(p)})=0$.

\begin{corollary}
  \label{acyclic p-local}
  Let $p$ be a prime and let $X$ be a connected space. If $(\Sigma X)_{(p)}\in\PP_p$, then $X$ is $p$-locally acyclic.
\end{corollary}

\begin{proof}
  Since $\mathrm{cat}\,(\Sigma X)_{(p)}\le 1$, it follows from Proposition \ref{FHL} that $\widetilde{H}^*(X;\Z/p)=0$. Since $(\Sigma X)_{(p)}$ is a homotopy retract of a $p$-local space of $p$-local finite type, $X$ itself is of $p$-local finite type. Thus $\widetilde{H}^*(X;\Z_{(p)})=0$.
\end{proof}

\begin{corollary}
  \label{acyclic}
  Let $Y$ be a connected space of finite type. If $\Sigma Y\in\PP$, then $Y$ is an acyclic space.
\end{corollary}

\begin{proof}
  By Lemma \ref{GPT local}, $(\Sigma Y)_{(p)}\in\PP_p$, and so by Corollary \ref{acyclic p-local}, $\widetilde{H}_*(Y;\Z_{(p)})=0$ for any prime $p$. Since $Y$ is of finite type, this implies $\widetilde{H}_*(Y;\Z)=0$, that is, $Y$ is an acyclic space.
\end{proof}

A map $f\colon X\to Y$, whose homotopy fibre is either an acyclic space or a disjoint union of such spaces, plays an important role in this paper.

\begin{example}
	For a connected space $X$, let $\kappa\colon TX\to X$ a map such that $TX$ is a space of type $K(\pi,1)$ and $\kappa$ induces an isomorphism in homology with respect to any local system of coefficients on $X$ (In particular in integral homomlogy). The existence of such a map is guaranteed by \cite{KaTh}. Furthermore, the homotopy fibre of $\kappa$ is an acyclic space.

\end{example}

We end this section by characterising such maps between connected spaces.

\begin{lemma}\label{lem-acyclic_fib}
Let $f\colon X\to Y$ be a map of connected spaces whose homotopy fibre is a disjoint union of acyclic spaces. Then $f$ can be factored as a composition
\[X\xto{l} X^+\xto{p} Y\]
where $X^+$ is the Quillen plus construction on $X$ with respect to some perfect normal subgroup of $\pi_1(X)$, and $p$ is a covering map. Conversely, the homotopy fibre of any map that can be factored this way is a disjoint union of acyclic spaces.
\end{lemma}

\begin{proof}
Suppose that $f\colon X\to Y$ is a map whose homotopy fibre $F$ is a disjoint union of acyclic spaces. If $F$ is connected, then $\pi_1(F)$ is perfect, and the image of the map induced by the fibre inclusion $P= \Ima(\pi_1(F)\to\pi_1(X))$ is a perfect normal subgroup of $\pi_1(X)$.  Hence $Y$ is homotopy equivalent to $X^+$ with respect to $P$. (See  \cite{D} for a much refined version of this statement.) Hence in this case the claim is clear, and we may assume $F$ has more than one connected component.

The exact sequence
\[\pi_1(X)\xto{f_*}\pi_1(Y)\to\pi_0(F)\to *\]
implies that $f_*$ is not onto. In particular $\pi_1(Y)\neq 1$. Let $H= \Ima(f_*)\le\pi_1(Y)$. One has a commutative diagram of fibrations:
\[\xymatrix{
\widehat{F}\ar[r]\ar[d]& \widehat{X}\ar[r]^{\widehat{f}}\ar[d] & \widetilde{Y}\ar[d]\\
F\ar[r]\ar[d] & X\ar[r]^f\ar[d]^\alpha & Y\ar[d]\\
\pi_1(Y)/H \ar[r] & BH\ar[r] & B\pi_1(Y)
}\]
where $\widetilde{Y}$ is the universal cover of $Y$, and $\widehat{X}$ is the covering of $X$ corresponding to the normal subgroup $K= \Ker(f_*)\le \pi_1(X)$. It follows that $\pi_1(\widehat{X})\cong K$ is a perfect group since $\widehat{F}$ is connected and acyclic, and that $\widetilde{Y} \simeq \widehat{X}^+$. Let $X^+$ denote the plus construction on $X$ with respect to $K$. Then $\pi_1(X^+) \cong\pi_1(X)/K = H$, and so  $\alpha\colon X\to BH$ factors through $X^+$. Expanding the composition into a diagram of fibrations,
\[\xymatrix{
A\ar[r]\ar@{=}[d] & \widehat{X}\ar[r]\ar[d] & \widetilde{X^+}\ar[d]\\
A\ar[r]\ar[d] & X\ar[r]^l\ar[d]^\alpha  & X^+\ar[d]^\beta\\
\ast\ar[r] & BH\ar[r]^{Bl_\ast} & BH
}\]
 it follows that $\widetilde{Y}\simeq \widehat{X}^+\simeq \widetilde{X^+}$. Since $\widetilde{Y}$ is the universal cover, it follows that $X^+ \simeq \widetilde{Y}/H$. Now, consider the covering $p\colon X^+ \simeq \widetilde{Y}/H\to \widetilde{Y}/\pi_1(Y)\simeq Y$. Since $\Ima(f_*) = H$, the map $f\colon X\to Y$ lifts  to $l'\colon X\to X^+$, whose homotopy fibre is acyclic with fundamental group $K$. This gives the factorisation of $f$ as claimed.

 The converse is clear.
 \end{proof}



\section{Basic properties of polyhedral products}
\label{Polyhedral}

In this section,we recall the definition of polyhedral products and prove some basic properties. Let $K$ denote a simplicial complex on a vertex set $[m]=\{1,2,\ldots,m\}$, and let $(\underline{X},\underline{A})=\{(X_i,A_i)\}_{i=1}^m$ denote a collection of pairs of spaces. If $(X_i,A_i)=(X,A)$ for all $i$, then we write $(X,A)$ for $(\underline{X},\underline{A})$. For $\sigma\subset[m]$, let
$$D(\sigma)= Y_1\times\cdots\times Y_m\quad\text{such that}\quad Y_i=\begin{cases}X_i&i\in\sigma\\A_i&i\not\in\sigma.\end{cases}$$
Then the \hadgesh{polyhedral product} of $(\underline{X},\underline{A})$ over $K$ is defined by
$$Z_K(\underline{X},\underline{A})=\bigcup_{\sigma\in K}D(\sigma).$$

The join of simplicial complexes $K_1$ and $K_2$ is defined by
$$K_1*K_2=\{\sigma_1\sqcup\sigma_2\,\vert\,\sigma_1\in K_1,\,\sigma_2\in K_2\}.$$
For a subset $V\subset[m]$, write $(\underline{X},\underline{A})_V=\{(X_i,A_i)\}_{i\in V}$. By the definition of polyhedral products, the following is immediate.

\begin{lemma}
  \label{join}
  Let $K_i$ be a simplicial complex with vertex set $V_i$ for $i=1,2$ such that $V_1\sqcup V_2=[m]$. Then
  $$Z_{K_1*K_2}(\underline{X},\underline{A})=Z_{K_1}(\underline{X},\underline{A})_{V_1}\times Z_{K_2}(\underline{X},\underline{A})_{V_2}.$$
\end{lemma}

For a non-empty subset $I$ of $[m]$, the full subcomplex of $K$ induced by $I$ is defined by
$$K_I=\{\sigma\in K\,\vert\,\sigma\subset I\}.$$

\begin{lemma}
  \label{full subcomplex}
  For a non-empty subset $I$ of $[m]$, $Z_{K_I}(\underline{X},\underline{A})_I$ is a retract of $Z_K(\underline{X},\underline{A})$.
\end{lemma}

\begin{proof}
  The space $\prod_{i\in I}X_i$ is a retract of $\prod_{i=1}^mX_i$, and one has
  	\[Z_{K_I}(\underline{X},\underline{A})_I = Z_K(\underline{X},\underline{A})\cap\prod_{i\in I}X_i.\]
  	Then the statement follows.
\end{proof}

The following is proved in \cite[Lemma 4.6]{K}.

\begin{lemma}
  \label{fibration lemma}
  Let $(\underline{F},\underline{F}^0)=\{(F_i,F_i^0)\}_{i=1}^m$ and $(\underline{E},\underline{E}^0)=\{(E_i,E_i^0)\}_{i=1}^m$, and assume that $(F_i,F_i^0)$ and $(E_i,E_i^0)$ are NDR-pairs and that $(F_i,F_i^0)\to(E_i,E_i^0)\to(B_i,B_i)$ is a homotopy fibration for $i=1,\ldots,m$. Then
  $$Z_K(\underline{F},\underline{F}^0)\to Z_K(\underline{E},\underline{E}^0)\to\prod_{i=1}^mB_i$$
  is a homotopy fibration.
\end{lemma}

For any pointed space $X$, let $CX$ denote the cone on $X$.

\begin{proposition}
  \label{fibration}
  Assume that $(X_i,A_i)$ is a pointed NDR-pair for each $i=1\ldots m$ and let $F_i$ denote the homotopy fibre of the inclusion $A_i\to X_i$. Then there is a homotopy fibration
  $$Z_K(C\underline{F},\underline{F})\to Z_K(\underline{X},\underline{A})\to\prod_{i=1}^mX_i,$$
  where $(C\underline{F},\underline{F})=\{(CF_i,F_i)\}_{i=1}^m$.
\end{proposition}

\begin{proof}
  Since $(CF_i,F_i)\to(X_i,A_i)\to(X_i,X_i)$ is a homotopy fibration such that $(CF_i,F_i)$ and $(X_i,A_i)$ are NDR-pairs for all $i$, the statement follows from Lemma \ref{fibration lemma}.
\end{proof}



\section{Properties of $Z_K(CX,X)$}
\label{Polyhedral-CX-X}

In this section we restrict attention to polyhedral products of pairs of the form $(CX, X)$ where $X$ is an arbitrary pointed space and $CX$ is the cone on $X$.


\subsection{Invariance}

Let  $(C\underline{X},\underline{X})=\{(CX_i,X_i)\}_{i=1}^m$. Recall from \cite[Theorem 2.21]{BBCG} that there is a natural homotopy decomposition
\begin{equation}
  \label{BBCG}
  \Sigma Z_K(C\underline{X},\underline{X})\simeq\Sigma\bigvee_{\emptyset\ne I\notin K}|\Sigma K_I|\wedge\widehat{X}^I,
\end{equation}
where $|L|$ means the geometric realization of a simplicial complex $L$ and $\widehat{X}^I=\bigwedge_{i\in I}X_i$.
	Let $\underline{X}$ and $\underline{Y}$ denote $\{X_i\}_{i\in[m]}$ and $\{Y_i\}_{i\in[m]}$, respectively. Let $f=\{(Cf_i,f_i)\colon(CX_i,X_i)\to(CY_i,Y_i)\}_{i=1}^m$ be a collection of maps. Assume that $X_i$ and $Y_i$ are connected, and that $\Sigma f_i$ is a homotopy equivalence for each $i=1,\ldots,m$. Let $Z_K(f)\colon Z_K(C\underline{X},\underline{X})\to Z_K(C\underline{Y},\underline{Y})$ be the map induced by $f$. Then  by \eqref{BBCG}, $\Sigma Z_K(f)$ is a homotopy equivalence. In particular $\Sigma Z_K(f)$ and hence $Z_K(f)$  both induce isomorphisms in homology. On the other hand, since $X_i$ and $Y_i$ is connected for each $i$, $Z_K(CX,X)$ and $Z_K(CY,Y)$ are simply-connected as in \cite{IK2}. Thus by the J.H.C. Whitehead theorem, the map $Z_K(f)$ is  a homotopy equivalence.

We generalise this observation to the case that $X_i$ and $Y_i$ are not necessarily connected. We start by considering some general properties of (homotopy) pushouts. For spaces $A$ and $B$, let
$$A\rtimes B= (A\times B)/(A\times*).$$

\begin{lemma}
  \label{pushout}
  Let $Z$ be the pushout space of the system
  \begin{equation}
    \label{hp lem}
    \xymatrix{B\times X & A\times X\ar[r]^{1\times\mathrm{incl}}\ar[l]_{f\times 1}&A\times CX.}
  \end{equation}
  If $f$ is null-homotopic, then there is a homotopy equivalence
  $$Z\xto{h}(B\rtimes X)\vee\Sigma(A\wedge X)$$
  which is natural with respect to $A,B,X$. Furthermore, the homotopy class of the homotopy equivalence $h$ depends only on the homotopy class of the null homotopy for $f$ as a map that restricts to $f$ on $A\times \{0\}$ and to the constant map on $A\times\{1\}$.
\end{lemma}
\begin{proof}
  Let $F$ be a null homotopy for $f$ and let $\bar{f}\colon CA\to B$ be the extension of $f$ to $CA$ via $F$. Then $f$ can be factored as the inclusion of $A$ in $CA$ followed by $\bar{f}$, and the pushout diagram of the system \eqref{hp lem} can be decomposed as:
  $$\xymatrix{A\times X\ar[r]^{1\times\mathrm{incl}}\ar[d]&A\times CX\ar[d]\\
  CA\times X\ar[r]^g\ar[d]_{\bar{f}\times 1}&A*X\ar[d]\\
  B\times X\ar[r]&Z}$$
  where $A*X$ is the join of $A$ and $X$, and both squares are homotopy pushout diagrams, the top by construction and the bottom because $g$ is a cofibration. The desired homotopy equivalence now follows since $g$ is null homotopic and $A*X\simeq\Sigma(A\wedge X)$. The equivalence is clearly natural with respect to $A, B$ and $X$ because of functoriality of homotopy pushouts.

  Finally, notice that the homotopy class of $\bar{f}$ as a map $CA\to B$ that is $f$ on $A\times\{0\}\subset CA$ and the constant map on $A\times \{1\}$ depends only on the homotopy class of the chosen null homotopy for $f$. This proves the last claim.
\end{proof}

Now we apply the above results on (homotopy) pushouts to polyhedral products.

\begin{lemma}
  \label{K+v}
  Let $K$ be a disjoint union of a simplicial complex $L$ with vertex set $[m-1]$ and an isolated vertex $\{m\}$. Then there is a homotopy equivalence
  $$Z_K(C\underline{X},\underline{X})\simeq(Z_L(C\underline{X},\underline{X})_{[m-1]}\rtimes X_m)\vee(\Sigma(\prod_{i=1}^{m-1}X_i)\wedge X_m)$$
  that is natural with respect to each $X_i$.
\end{lemma}

\begin{proof}
	Consider $\prod_{i=1}^{m-1}X_i$ as $Z_L(\underline{X},\underline{X})$. Then one has an obvious inclusion
		\[f\colon \prod_{i=1}^{m-1}X_i\to Z_L(C\underline{X},\underline{X})_{[m-1]}.\] We claim that $f$ is null homotopic. To see this, notice that for each $i=1,\ldots, m-1$, the map $f$ factors through $\prod_{j=1}^{i-1}X_j\times CX_i\times \prod_{j=i+1}^{m-1}X_j$, and hence up to homotopy through the projection to $\prod_{j=1}^{i-1}X_j\times *\times\prod_{j=i+1}^{m-1}X_j$. Proceeding by induction, removing one factor at a time, the claim follows.
  Next, consider the pushout diagram:
  $$\xymatrix{\prod_{i=1}^{m-1}X_i\times X_m\ar[r]^{1\times \mathrm{incl}}\ar[d]_{f\times 1}&\prod_{i=1}^{m-1}X_i\times CX_m\ar[d]\\
  Z_L(C\underline{X},\underline{X})_{[m-1]}\times X_m\ar[r]&Z_K(C\underline{X},\underline{X})}$$
  The statement now follows from  Lemma \ref{pushout}.
  \end{proof}

\begin{proposition}
  \label{invariance}
  Let $(C\underline{X},\underline{X})=\{(CX_i,X_i)\}_{i=1}^m$ and $(C\underline{Y},\underline{Y})=\{(CY_i,Y_i)\}_{i=1}^m$. Let $f=\{(Cf_i,f_i)\colon(CX_i,X_i)\to(CY_i,Y_i)\}_{i=1}^m$ be be a collection of maps. If $\Sigma f_i$ is a homotopy equivalence for each $i=1,\ldots,m$, then the map
  $$Z_K(f)\colon Z_K(C\underline{X},\underline{X})\to Z_K(C\underline{Y},\underline{Y})$$
  induced by $f$ on polyhedral products is a homotopy equivalence.
\end{proposition}

\begin{proof}
   Let $L\subseteq K$ be a subcomplex. Let $\overline{L}$ be the simplicial complex obtained as the disjoint union of $L$ and  the set of all vertices of $K$ that are not in $L$. The link and the deletion of a vertex $v$ of $K$ are defined by
  $$\lk_K(v)=\{\sigma\in K\,\vert\,v\not\in\sigma,\,\sigma\cup v\in K\}\quad\text{and}\quad\dl_K(v)=\{\sigma\in K\,\vert\,v\not\in\sigma\}.$$
  Then there is a pushout of simplicial complexes
  $$\xymatrix{\overline{\lk_K(v)}\ar[r]\ar[d]&\overline{\lk_K(v)\ast v}\ar[d]\\
  \overline{\dl_K(v)}\ar[r]&K}$$
  which induces a pushout of spaces
  \begin{equation}
    \label{pushout Z}
    \xymatrix{Z_{\overline{\lk_K(v)}}(C\underline{A},\underline{A})\ar[r]\ar[d]&Z_{\overline{\lk_K(v)\ast v}}(C\underline{A},\underline{A})\ar[d]\\
    Z_{\overline{\dl_K(v)}}(C\underline{A},\underline{A})\ar[r]&Z_K(C\underline{A},\underline{A})}
  \end{equation}
  for any $(C\underline{A},\underline{A})=\{(CA_i,A_i)\}_{i=1}^m$.

  We prove the statement by induction on $m$ - the number of vertices of $K$. For $m=1$, $Z_K(C\underline{X},\underline{X})=CX_1\simeq CY_1=Z_K(C\underline{Y},\underline{Y})$, and so the statement is true. Assume that the statement holds for any simplicial complex on at most $m-1$ vertices. Consider a commutative diagram
  $$\xymatrix{Z_{\overline{\dl_K(m)}}(C\underline{X},\underline{X})\ar[d]&Z_{\overline{\lk_K(m)}}(C\underline{X},\underline{X})\ar[r]\ar[l]\ar[d]&Z_{\overline{\lk_K(m)*m}}(C\underline{X},\underline{X})\ar[d]\\
  Z_{\overline{\dl_K(m)}}(C\underline{Y},\underline{Y})&Z_{\overline{\lk_K(m)}}(C\underline{Y},\underline{Y})\ar[r]\ar[l]&Z_{\overline{\lk_K(m)*m}}(C\underline{Y},\underline{Y})}$$
  where all horizontal maps are cofibrations. Since \[\overline{\dl_K(m)}=\dl_K(m)\sqcup m\quad \text{and}\quad \overline{\lk(m)}=\dl_{\overline{\lk(m)}}(m)\sqcup m,\]
  the left and the middle vertical maps are homotopy equivalences by Lemma \ref{K+v} and the induction hypothesis. If $\lk_K(m)$ has $m-1$ vertices, then $\overline{\lk_K(m)*m}=\lk_K(m)*m$, and so by Lemma \ref{join}, the right vertical map is a homotopy equivalence. If $\lk_K(m)$ has  less than $m-1$ vertices, then there is a vertex $v$ of $K$ such that $\overline{\lk_K(m)*m}=\dl_{\overline{\lk_K(m)*m}}(v)\sqcup v$. Then the right vertical map is a homotopy equivalence by Lemma \ref{K+v} and the induction hypothesis. Thus by homotopy invariance of homotopy pushouts applied to \eqref{pushout Z} the statement follows.
\end{proof}

\begin{corollary}\label{cor-acyclic}
	Let $X$ be a disjoint union of acyclic spaces, and let $S=\pi_0(X)$. Then for any simplicial complex $K$
  $$Z_K(CX, X)\simeq Z_K(CS,S).$$
\end{corollary}
\begin{proof}
	Under our hypothesis the natural projection $X\to S=\pi_0(X)$ is a homotopy equivalence after suspension. The corollary follows from Proposition \ref{invariance}.
\end{proof}


\subsection{The $p$-local and the integral cases}

Let $p$ be a fixed prime. We first consider the $p$-local case. A subset $M$ of $[m]$ with $|M|\ge 2$ is called a \hadgesh{minimal non-face} of $K$ if it is not a simplex of $K$ and all proper subsets of $M$ are simplices of $K$.

\begin{lemma}
  \label{discrete}
  Let $X$ be a disjoint union of spaces of finite $p$-local type, and assume that $K$ is not a simplex. If the universal cover of $Z_K(CX,X)$ belongs to $\PP_p$, then each connected component $X_0$ of $X$ is $p$-locally acyclic.
\end{lemma}

\begin{proof}
  Since $K$ is assumed not to be a  simplex, it has a minimal non-face  $M$. Let $L$ be the boundary of a simplex with vertex set $M$. Then $L=K_M$, and so by Lemma \ref{full subcomplex}, $Z_L(CX,X)$ is a retract of $Z_K(CX,X)$. It is easy to see that
  $$Z_L(CX,X)=\underbrace{X*\cdots*X}_{|M|}\simeq\Sigma^{|M|-1}X^{\wedge|M|}.$$
 Let $X_0$ be a path component of $X$. In particular $X_0$ is a retract of $X$, and so   $\Sigma^{|M|-1}X_0^{\wedge|M|}$ is a homotopy retract of $Z_L(CX,X)$ and hence of $Z_K(CX,X)$. Since $\Sigma^{|M|-1}X_0^{\wedge|M|}$ is simply-connected, it is also a homotopy retract of the universal cover of $Z_K(CX,X)$, hence $\Sigma^{|M|-1}X_0^{\wedge|M|}\in\PP_p$. Thus by Corollary \ref{acyclic p-local}, $X_0^{\wedge|M|}$ is $p$-locally acyclic and by the K\"unneth formula $X_0$ is $p$-locally acyclic, as claimed.
\end{proof}

A flag complex is a simplicial complex $K$ that is generated by its 1-skeleton, namely where for every clique in the 1-skeleton of $K$ there is a simplex in $K$. Thus a simplicial complex $K$ is a flag complex if and only if all of its minimal non-faces have cardinality two. The following lemma is proved in \cite[Proposition 4.2]{K} (cf. \cite{BBC}). We note that there is a terminology error in the proof of \cite[Proposition 4.2]{K}: "acyclic space" should be replaced with "aspherical space".

\begin{lemma}
  \label{K(pi,1)}
  Let $(\underline{X},*)=\{(X_i,*)\}_{i=1}^m$. Then $Z_K(\underline{X},*)$ is a non-trivial $K(\pi,1)$ if and only if each $X_i$ is a non-trivial $K(\pi,1)$ and $K$ is a flag complex.
\end{lemma}

\begin{theorem}
  \label{Z(CX,X) p-local}
  If $X$ is a disjoint union of  spaces of finite $p$-local type, then the following are equivalent:
  \begin{enumerate}
    \item the universal cover of $Z_K(CX,X)$ belongs to $\PP_p$;
    \item $Z_K(CX,X)$ is a $K(\pi,1)$;
    \item either of the following conditions holds:
    \begin{enumerate}
      \item $K$ is a simplex;
      \item $X$ is $p$-locally acyclic and $K$ is not a simplex;
      \item $X$ is a disjoint union of $p$-locally acyclic spaces and $K$ is a flag complex which is not a simplex.
    \end{enumerate}
  \end{enumerate}
\end{theorem}

\begin{proof}
  Suppose that $Z_K(CX,X)$ belongs to $\PP_p$.  If $K$ is not a simplex, then by Lemma \ref{discrete}, each connected component of $X$ is $p$-locally acyclic. Suppose that $X$ has more than a single connected component. Then $S^0$ is a retract of $X$. If $K$ is not a flag complex, then it has a minimal non-face $M$ of cardinality $\ge 3$.  By analogy to the proof of Lemma \ref{discrete}, $S^{|M|-1}$, which is simply connected for $|M|\geq 3$,  is a homotopy retract of the universal cover of $Z_K(CX,X)$, which is a $p$-local space of finite $p$-local type by hypothesis. Bu the sphere $S^{|M|-1}$ is not $p$-local, and hence cannot be a homotopy retract of a $p$-local space. This gives a contradiction, and so   $K$ must be a flag complex. This shows that (1) implies (3).

  Next we show that (3) implies (2). If $K$ is a simplex, then $Z_K(CX,X)\simeq*$ and the claim holds vacuously. By hypothesis $X$ is of finite $p$-local type, and hence $\Sigma X$ is $p$-local. Thus, if $X$ is $p$-locally acyclic, then $\Sigma X$ is contractible. Therefore, in this case $Z_K(CX,X)\simeq*$  by Proposition \ref{invariance}. Suppose now that $X$ is a disjoint union of $p$-locally acyclic spaces and that $K$ is a flag complex. Since the suspension of every connected component of $X$ is contractible, Proposition \ref{invariance} implies that we may assume $X$ is discrete.  With this assumption, let $G$ be a discrete group with the same cardinality as that of $X$. Then $Z_K(CX,X)\cong Z_K(CG,G)$, and by Lemma \ref{fibration}, there is a homotopy fibration
  $$Z_K(CX,X)\to Z_K(BG,*)\to BG^m.$$
By Lemma \ref{K(pi,1)},  $Z_K(BG,*)$ is a $K(\pi,1)$ by Lemma \ref{K(pi,1)}. Thus $Z_K(CX,X)$ is a $K(\pi,1)$, and hence (3) implies (2). Clearly, (2) implies (1), and so the proof is complete.

\end{proof}

\begin{remark}\label{Rem-discrete-group-trick}
	The fact that the homotopy type of $Z_K(CG, G)$ depends only on the cardinality of $G$ and not on the isomorphism class of $G$ is the main ingredient in the proof of Theorem \ref{application} in Section \ref*{Groups}
\end{remark}

We next consider the integral analog of Theorem \ref{Z(CX,X) p-local}.

\begin{theorem}
  \label{Z(CX,X)}
  If $X$ is a disjoint union of spaces of finite type, then the following are equivalent:
  \begin{enumerate}
    \item $Z_K(CX,X)$ belongs to $\PP$;
    \item $Z_K(CX,X)$ is a $K(\pi,1)$;
    \item either of the following conditions holds:
    \begin{enumerate}
      \item $K$ is a simplex;
      \item $X$ is an acyclic space and $K$ is not a simplex;
      \item $X$ is a disjoint union of acyclic spaces and $K$ is a flag complex that is not a simplex.
    \end{enumerate}
  \end{enumerate}
\end{theorem}

\begin{proof}
  If $Z_K(CX,X)$ belongs to $\PP$, then in particular the $p$-localization of its universal cover belongs to $\PP_p$ for every prime $p$, and hence by Lemma \ref{discrete} each path component of $X$ is $p$-locally acyclic for all primes $p$. Since each connected component of $X$ is of finite type, this implies that $X$ is a disjoint union of acyclic spaces. Hence similarly to the proof of Theorem \ref{Z(CX,X) p-local}, it follows that  (1) implies (3). The for (3) implying (2) is done in the same way as the proof of Theorem \ref{Z(CX,X) p-local}. Clearly, (2) implies (1). Thus the proof is complete.
\end{proof}


\subsection{The rational case}

Finally we consider the rational case. We begin with a simple lemma.

\begin{lemma}
  \label{wedge of spheres}
  If $X,Y$ are suspensions such that $X$ is simply-connected and $X,Y$ are not rationally contractible, then $X\vee Y$ does not belong to $\PP_0$.
\end{lemma}

\begin{proof}
  If $Y$ is a $K(\pi,1)$, then $\pi$ is infinite, and so $\pi_i(X\vee Y)\otimes\mathbb{Q}$ is not a finite dimensional $\mathbb{Q}$-vector space for some $i$. Thus $X\vee Y\not\in\PP_0$.

  Suppose $Y$ is not a $K(\pi,1)$ and $X\vee Y\in\PP_0$. Then $X\vee Y$ is a homotopy retract of a space $Z$ whose universal cover admits a finite type rational generalised Postnikov tower of finite length. By the rational analog of \cite[Proposition 4.6]{IK}, the length of non-trivial iterated Whitehead products in $\pi_*(Z)$ is bounded by the length of of a rational generalised Postnikov tower for $Z$. On the other hand, $\pi_*(X\vee Y)$ has rationally non-trivial iterated Whitehead products of arbitrary length, implying so does $\pi_*(Z)$. This is a contradiction, and so $X\vee Y\not\in\PP_0$.
\end{proof}

The following is proved in \cite[Proposition 4.1]{HST}.

\begin{lemma}
  \label{mutually disjoint}
  Suppose $K$ has minimal non-faces $M_1,M_2$ such that $M_1\cap M_2\ne\emptyset$. Then $\Sigma^{|M_1|-1}X^{\wedge|M_1|}\vee\Sigma^{|M_2|-1}X^{\wedge|M_2|}$ is a homotopy retract of $Z_K(CX,X)$.
\end{lemma}

\begin{theorem}
  \label{Z(CX,X) rational}
  If each connected component of $X$ is of finite rational type, then the following are equivalent:
  \begin{enumerate}
    \item $Z_K(CX,X)$ belongs to $\PP_0$;
    \item $Z_K(CX,X)$ is either a $K(\pi,1)$, or of the homotopy type of a finite product of rational spheres of dimension $>1$;
    \item one of the following conditions holds:
    \begin{enumerate}
      \item $K$ is a simplex;
      \item $X$ is rationally acyclic and $K$ is not a simplex;
      \item $X$ is a disjoint union of rationally acyclic spaces and $K$ is a flag complex that is not a simplex;
      \item $X$ is a homology rational sphere and minimal non-faces of $K$ are mutually disjoint.
    \end{enumerate}
  \end{enumerate}
\end{theorem}

\begin{proof}
  Suppose that $Z_K(CX,X)\in\PP_0$ and that  $K$ is not a simplex. Then $Z_K(CX,X)$ is a homotopy retract of a space $Y$, whose universal cover admits a finite type rational generalised Postnikov tower of finite length.  Since $K$ is not a simplex, it has a minimal non-face $M$. Hence, as in the proof of Lemma \ref{discrete}, $\Sigma^{|M|-1}X^{\wedge|M|}$ is a homotopy retract of $Z_K(CX,X)$, and so $\Sigma^{|M|-1}X^{\wedge|M|}\in\PP_0$. If $X_0$ is any connected component of $X$, then $\Sigma^{|M|-1}X^{\wedge |M|}_0$ is a retract of $\Sigma^{|M|-1}X^{\wedge|M|}$ and hence also in $\PP_0$. But since $|M|\geq 2$, the space $\Sigma^{|M|-1}X^{\wedge |M|}_0$ is a simply-connected suspension.

If $X$ is connected then $\Sigma^{|M|-1}X^{\wedge|M|}$ is a simply-connected rational suspension space. Hence by \cite[Theorem 24.5]{FHT}, if $X$ is not rationally acyclic, then it is a wedge of rational spheres of dimension $>1$. If the number of rational spheres in that wedge is larger than one, then for some $m,n>1$ the wedge $S^m_{(0)}\vee S^n_{(0)}$  is a homotopy retract of $\Sigma^{|M|-1}X^{\wedge|M|}$, implying $S^m_{(0)}\vee S^n_{(0)}\in\PP_0$. By Lemma \ref{wedge of spheres}, this is a contradiction. Hence  $X$ is either a single rational homology sphere, as in (3)(d) or rationally acyclic, as in (3)(b). In the first case, if $K$ has minimal non-faces $M_1,M_2$ such that $M_1\cap M_2\ne\emptyset$, then by Lemma \ref{mutually disjoint},  a wedge of rational spheres of dimension $>1$ is a homotopy retract of $Z_K(CX,X)$. However, this is impossible by Lemma \ref{wedge of spheres}, and thus minimal non-faces of $K$ are mutually disjoint.

Suppose $X$ has more than one connected component. As before, since $S^0$ is a retract of $X$, $S^{|M|-1}$ is a homotopy retract of $Z_K(CX,X)$.  If $K$ is not a flag complex, then $|M|\geq 3$. Thus $S^{|M|-1}$ is simply connected and hence a homotopy retract of the universal cover of  $Z_K(CX,X)$. Since $Z_K(CX,X)\in\PP_0$, its universal cover is a rational space. But $S^{|M|-1}$ is not a rational space and we obtain a contradiction. Thus $|M|=2$, and it follows that $K$ is a flag complex.
Let $X_0$ be a connected component of $X$. Since $X$ is not connected, $S^0\vee X_0$ is a retract of $X$. It follows easily that $S^1\vee\Sigma X_0^{\wedge 2}$ is a retract of $\Sigma X^{\wedge 2}$, and hence that $S^1\vee\Sigma X_0^{\wedge 2}\in\PP_0$. If $X_0$ is not rationally acyclic, then $\Sigma X_0^{\wedge 2}$ is a simply-connected rationally non-trivial suspension. Thus one gets a contradiction to Lemma \ref{wedge of spheres}, hence any connected component of $X$ must be rationally acyclic, as in (3)(c). The proof  that (1) implies (3) is thus complete.

  If one of the conditions (3) (a), (b) or (c) holds, then $Z_K(CX,X)$ is a $K(\pi,1)$ by the same argument as the proof of Theorem \ref{Z(CX,X) p-local}. Suppose that condition (3)(d) holds. Let $M_1,\ldots,M_r$ be all minimal non-faces of $K$. For a finite set $S$, let $\partial\Delta^S$ denote the boundary of a full simplex over $S$. Since $M_1\ldots,M_r$ are mutually disjoint, we can consider a simplicial complex $L=\partial\Delta^{M_1}*\cdots*\partial\Delta^{M_r}$. Take any subset $\sigma\subset[m]$ such that $\sigma\cap M_i\ne M_i$ for each $i$. Then $\sigma$ is a simplex of $L$, implying that all minimal non-faces of $L$ are $M_1,\ldots,M_r$. Since simplicial complexes are determined by their minimal non-faces, one gets $K=L$.  By Lemma \ref{join},
  \[Z_K(CX,X)\simeq Z_{\partial\Delta^{M_1}}(CX,X)\times\cdots\times Z_{\partial\Delta^{M_r}}(CX,X).\]
  Since $Z_{\partial\Delta^{M_i}}(CX,X)\simeq\Sigma^{|M_i|-1}X^{\wedge|M_i|}$ and $X$ is assumed to be a rational homology sphere, $Z_K(CX,X)$ is of the homotopy type of a product of rational spheres of dimension at least 2.  
  Thus (3) implies (2). Clearly, (2) implies (1), and therefore the proof is complete.
\end{proof}



\section{Proofs of the main theorems}
\label{Main}

We are now ready to restate and prove Theorem \ref{main}.

\begin{theorem}
  \label{main rephrase}
 Let $K$ be a simplicial complex, and let $(X,A)$ be an NDR-pair. Let $F$ denote the homotopy fibre of the inclusion $A\to X$, and assume that each component of $F$ is of finite type. Then $Z_K(X,A)$ belongs to $\PP$ if and only if one of the following conditions holds:
 \begin{enumerate}
 	\item $X\in\PP$ and $K$ is a simplex;\label{Thm-main-body-condition-1}
 	\item $X\in\PP$, $F$ is an acyclic space, and $K$ is not a simplex;\label{Thm-main-body-condition-2}
 	\item $X\in\PP$, $F$ is a disjoint union of acyclic spaces, and $K$ is a flag complex that is not a simplex. \label{Thm-main-body-condition-3}
 \end{enumerate}
 Moreover, the statement holds for $\PP$ replaced by $\F$ or $\K$ throughout.
\end{theorem}

\begin{proof}
If $X\in \PP$ and $K$ is a simplex, then $Z_K(X,A) = X^m$, so $Z_K(X,A)\in\PP$ by Lemma \ref{lem-products}. Thus assume that $K$ is not a simplex and that either the condition  \eqref{Thm-main-body-condition-2} or \eqref{Thm-main-body-condition-3} is satisfied. Then, by Theorem \ref{Z(CX,X)}, $Z_K(CF, F)$ is a $K(\pi,1)$.
By Proposition \ref{fibration} one has a fibration
  \begin{equation}
\label{fibrationfinal}
Z_K(CF,F)\to Z_K(X,A)\to X^m.
\end{equation}
Let $\widetilde{X}$ denote the universal cover of $X$, and let $Z$  denote the universal cover of $Z_K(X,A)$.
From the homotopy sequence for \eqref{fibrationfinal} it follows that $Z_K(X,A)$ and $X^m$ have homotopy equivalent 2-connected covers. Thus one has a homotopy commutative square
\begin{equation}\xymatrix{
Z\ar[d]\ar[r] & \widetilde{X}^m\ar[d]\\
K(\pi_2(Z),2) \ar[r] & K(\pi_2(\widetilde{X}^m),2)
}
\label{pullback-Universal-cover}
\end{equation}
where the vertical maps are the classifying maps for the 2-connected covers. Thus the diagram is a homotopy pullback square.

Let \[G=\mathrm{Im}\{\pi_2(X^m)\xto{\delta}\pi_1(Z_K(CF,F))\} \cong \coKer\{\pi_2(Z)\to\pi_2(\widetilde{X}^m)\}.\]
It follows that the bottom horizontal map in \eqref{pullback-Universal-cover} is a principal fibration with fibre $K(G,1)$ and hence so is the induced top horizontal map, and by delooping it we obtain a homotopy fibration:
$$Z\to\widetilde{X}^m\xrightarrow{g}K(G,2).$$
Since $\widetilde{X}^m\in\PP$, it follows from Proposition \ref{P} that $Z\in\PP$. Thus $Z_K(X,A)\in\PP$ as desired. The statement for $\F,\K$ follows from the homotopy fibration \eqref{fibrationfinal} and Lemma \ref{F-K}.

Conversely, suppose that $Z_K(X,A)\in\PP$. Since $X$ is a retract of $Z_K(X,A)$ and $\PP$ is closed under homotopy retracts, $X\in\PP$. Hence $X^m\in\PP$ by Lemma \ref{lem-products}, and it follows from \eqref{fibrationfinal} and Proposition \ref{P} that $Z_K(CF,F)\in\PP$. The claim now follows at once from Theorem \ref{Z(CX,X)}. The proof for $\F,\K$ follows similarly by using Lemma \ref{F-K}.
\end{proof}

The proofs of Theorems \ref{main p-local} for $\PP_p$ and \ref{main rational} for $\PP_0$ follows precisely the same line of argument, except one uses Theorems \ref{Z(CX,X) p-local} and \ref{Z(CX,X) rational} instead of Theorem \ref{Z(CX,X)}. Note that one has $p$-local and rational analogs of Lemma \ref{F-K}. Thus the proofs for $\F_p,\K_p,\F_0,\K_0,\E_0$ follow similarly.



\newcommand{\GG}{\mathbf{G}}
\newcommand{\HH}{\mathbf{H}}

\section{Graph products of groups}
\label{Groups}

We are now ready to prove Theorem \ref{application}. For any simple graph $\Gamma$, let $\Fl(\Gamma)$ denote the associated flag complex. Recall that for a graph $\Gamma$ with vertex set $[m]$, and a collection of groups $\GG = \{G_i\}_{i=1}^m$, the graph product $\GG^\Gamma$ is defined to be
\[\mathbf{G}^\Gamma= \langle G_1, \ldots ,G_m]\;|\; [G_i,G_j]=1\;\forall\{i,j\}\in E\rangle,\]
and $K_\Gamma(\mathbf{G})$ denotes the kernel of the canonical map $\mathbf{G}^\Gamma\to \prod_{i=1}^m G_i$.

\begin{lemma}
  \label{BG}
  Let $\Gamma$ be a simple graph with vertex set $[m]$ and $\GG=\{G_i\}_{i=1}^m$ be a collection of discrete groups. Let $B\GG$ denote the collection $\{BG_i\}_{i=1}^m$. Then there is a homotopy equivalence
  $$B(\GG^\Gamma)\simeq Z_{\Fl(\Gamma)}(B\GG,*).$$
\end{lemma}
\begin{proof}
	For any collection of pointed spaces $\mathbf{X}$, one can easily calculate the fundamental group of $Z_K(\mathbf{X},*)$ using the van Kampen theorem as in \cite[Proposition 1.2]{S} and \cite[Lemma 4.3]{K} such that $\pi_1(Z_{\Fl(\Gamma)}(B\GG,*))\cong \GG^\Gamma$. The lemma now follows from Lemma \ref{K(pi,1)}.
\end{proof}

\begin{proposition}
  \label{K(G)}
  Let $\Gamma$ be a simple graph with vertex set $[m]$ and let  $\GG=\{G_i\}_{i=1}^m$ be a collection of discrete groups. There is a homotopy equivalence
  $$BK_\Gamma(\GG)\simeq Z_{\Fl(\Gamma)}(C\GG,\GG),$$
  where, on the right hand side, each $G_i$ is regarded as a discrete set of points.
\end{proposition}

\begin{proof}
 The calculation of the fundamental group of $Z_K(X,*)$ in \cite[Proposition 1.2]{S} and \cite[Lemma 4.3]{K} implies that the inclusion $Z_{\Fl(\Gamma)}(B\GG,*)\to \prod_{i=1}^mBG_i$ together with Lemma \ref{BG} induces the canonical surjection $\GG^\Gamma\to\prod_{i=1}^mG_i$ on fundamental groups. By Proposition \ref{fibration}, there is a homotopy fibration
  $$Z_{\Fl(\Gamma)}(C\GG,\GG)\to Z_{\Fl(\Gamma)}(B\GG,*)\to\prod_{i=1}^mBG_i,$$
  where the second map is the inclusion. Thus the claim follows.
\end{proof}

\begin{theorem}
  \label{main K(G)}
  Let $\GG=\{G_i\}_{i=1}^m$ and $\HH=\{H_i\}_{i=1}^m$ be collections of discrete groups and $\Gamma$ be a simple graph with vertex set $[m]$. If $|G_i|=|H_i|$ for each $i$, then
  $$K_\Gamma(\GG)\cong K_\Gamma(\HH).$$
\end{theorem}

\begin{proof}
  If $|G_i|=|H_i|$ for each $i$, then $Z_{\Fl(\Gamma)}(C\GG,\GG)\cong Z_{\Fl(\Gamma)}(C\HH,\HH)$. The statement follows at once by Proposition \ref{K(G)}.
\end{proof}


\begin{thebibliography}{99}
  \bibitem{A} D. Anick, Connections between Yoneda and Pontrjagin algebras, Algebraic Topology, Aarhus 1982, \emph{Springer-Verlag Lecture Notes in Math.} \textbf{1051} (1984), 331-350.

  \bibitem{BBC} A. Bahri, M. Bendersky, F.R. Cohen, Polyhedral products and features of their homotopy theory, \emph{Handbook of Homotopy Theory}, pp. 105-146, CRC Press/Chapman and Hall Handbooks in Mathematics Series, 2019.

  \bibitem{BBCG} A. Bahri, M. Bendersky, F.R. Cohen, and S. Gitler, The polyhedral product functor: a method of decomposition for moment-angle complexes, arrangements and related spaces, \emph{Advances in Math.} \textbf{225} (2010), 1634-1668.

  \bibitem{BBCG1} A. Bahri, M. Bendersky, F.R. Cohen, and S. Gitler, On the rational type of moment-angle complexes, \emph{Proc. Steklov Inst. Math.}\textbf{286} (2014), no. 1, 219-223.

  \bibitem{DJ} M.W. Davis and T. Januszkiewicz, Convex polytopes, Coxeter orbifolds and torus actions, \emph{Duke Math. J.} \textbf{62} (1991), 417-451.

\bibitem{D}  E. Dror, Acyclic spaces, \emph{Topology} \textbf{11} (1972), 339-348.

\bibitem{CFFS} W. Chach\'olski,  E. D. Farjoun, Emmanuel, R. Flores,  J. Scherer, Cellular properties of nilpotent spaces, \emph{Geom. Topol.} \textbf{19} (2015), no. 5, 2741-2766.

  \bibitem{FHL} Y. F\'elix, S. Halperin, and J.-M. Lemaire, Mod $p$ loop space homology, \emph{Invent. math.} \textbf{95} (1989), no. 2, 247-262.

\bibitem{FHT1} Y. F\'elix, S. Halperin, and J.-C. Thomas, \emph{The homotopy Lie algebra for finite complexes}, Inst. Hautes \'Etudes Sci. Publ. Math. No. 56 (1982), 179-202

  \bibitem{FHT} Y. F\'elix, S. Halperin, and J.-C. Thomas, \emph{Rational Homotopy Theory}, Graduate Texts in Mathematics \textbf{205}, Springer-Verlag, 2000.


  \bibitem{GT2} J. Grbi\'c and S. Theriault, The homotopy type of the polyhedral product for shifted complexes, \emph{Advances in Math.} \textbf{245} (2013), 690-715.

  \bibitem{GW} V. Gruji\'c and V. Welker, Discrete Morse theory for moment-angle complexes of pairs $(D^n,S^{n-1})$, \emph{Monatsh. Math.} \textbf{176} (2015), no. 2, 255-273.


  \bibitem{HST} Y. Hao, Q. Sun, and S. Theriault, Moore's conjecture for polyhedral products, \emph{Math. Proc. Cambridge Phil. Soc.} \textbf{167} (2019), no. 1, 23-33.

  \bibitem{IK} K. Iriye and D. Kishimoto, Postnikov towers with fibers generalized Eilenberg-MacLane spaces, \emph{Homology Homotopy Appl.} \textbf{16} (2014), no. 1, 139-157.

  \bibitem{IK1} K. Iriye and D. Kishimoto, Decompositions of polyhedral products for shifted complexes, \emph{Advances in Math.} \textbf{245} (2013), 716-736.

  \bibitem{IK2} K. Iriye and D. Kishimoto, Fat-wedge filtrations and decomposition of polyhedral products, \emph{Kyoto J. Math.} \textbf{59} (2019), no. 1, 1-51.



  \bibitem{KT} Y. Kamiyama and S. Tsukuda, The configuration space of the $n$-arms machine in the Euclidean
space, \emph{Topology Appl.} \textbf{154} (2007), 1447-1464.


  \bibitem{KaTh} D. Kan and W. Thurston, Every connected space has the homology of a $K(\pi,1)$, \emph{Topology} \textbf{15} (1976), 253-258.

  \bibitem{Kim} S.-H. Kim, Surface subgroups of graph products of groups, \emph{Internat. J. Algebra Comput.} \textbf{22} (2012), no. 8, 1240003, 20 pp.

  \bibitem{K} D. Kishimoto, Right-angled Coxeter quandles and polyhedral products, \emph{Proc. Amer. Math. Soc.} \textbf{147} (2019), no. 9, 3715-3727.

  \bibitem{KL} D. Kishimoto and R. Levi, Polyhedral products over finite posets, \url{arXiv:1903.07951}.


  \bibitem{MS} J.C. Moore and L. Smith, Hopf algebras and multiplicative fibrations, I, \emph{Amer. J. Math.} \textbf{90} (1968), 752-780.

  \bibitem{P} G. Porter, The homotopy groups of wedges of suspensions, \emph{Amer. J. Math.} \textbf{88} (1966), 655-663.

  \bibitem{S} M. Stafa, On the fundamental group of certain polyhedral products, \emph{J. Pure Appl. Alg.} \textbf{219} (2015), 2279-2299.
\end{thebibliography}
\end{document}